\documentclass[11pt]{amsart}
\usepackage{latexsym}
\usepackage{amssymb}
\usepackage{amsfonts}
\usepackage{graphicx}
\usepackage{epsf}
\usepackage[all]{xypic}
\usepackage{delarray}
\usepackage{setspace}

\newtheorem{theorem}{Theorem}[section]
\newtheorem{corollary}[theorem]{Corollary}

\newtheorem{proposition}[theorem]{Proposition}

\newcommand{\Hom}{{\rm Hom}}

\definemorphism{dato}\dashed \tip \notip
\definemorphism{Dato}\Dashed \tip \notip

\newcommand{\Ss}{\mathcal{S}}

\begin{document}

\sloppy

\title[Faith's QF conjecture]{Semiprimary selfinjective algebras with at most countable dimensional Jacobson quotient are QF}


\begin{abstract}
We give a positive solution to a conjecture of Faith stating that a self-injective semiprimary ring is QF, for algebras which are at most countable dimensional modulo their Jacobson radical. As a consequence of the method used, we also give short proofs of several other known positive answers to this conjecture.
\end{abstract}

\author{M.C. Iovanov}
\address{University of Southern California \\
3620 S Vermont Ave, KAP 108 \\
Los Angeles, CA 90089, USA}
\address{University of Bucharest, Facultatea de Matematica\\ Str.
Academiei 14, Bucharest 1, RO-010014, Romania}
\address{e-mail: iovanov@usc.edu, yovanov@gmail.com}

\date{}

\maketitle

\section{Introduction and Preliminaries}

In classical ring theory, among the rings of interest and intensively studied in literature are the left or right selfinjective rings. Left selfinjective rings which are also artinian form another important class of rings called quasi-Frobenius (QF) rings. There are many equivalent definitions of these rings, and they have an intrinsic symmetry: a ring is QF if it is right selfinjective and right semiartinian, or equivalently, noetherian or artinian on one side and injective on one side. Classical results also include those of Faith and Walker stating that such rings are characterized by the fact that all right (equivalently, all left) injective (equivalently, projective) modules are projective (injective). These rings are important generalizations of Frobenius algebras, retaining the categorical properties of these; examples include group algebras of finite groups, Hopf algebras, certain cohomology rings. Moreover, such rings are important in many fields of mathematics, from representation theory, category theory, homological algebra and topology to coding theory. 

Perhaps one of the most interesting questions regarding QF rings, and also in ring theory in general, is the following question, known in literature as Faith's QF conjecture: 

\vspace{.5cm}

{\bf Conjecture [Faith]}\\
\emph{A left selfinjective semiprimary ring is QF.}

\vspace{.5cm}

Much work has been dedicated to this problem over the years \cite{ANY, CH, CS, NY, NY2, NY3, O, X}; we also refer to the recent survey \cite{FH} which contains a comprehensive account of the history and known results on QF rings and related topics. 

In this note we present a positive answer for algebras $A$ over a field $K$, which are at most countable dimensional modulo their Jacobson radical, i.e. $A/Jac(A)$ is at most $\aleph_0$. This includes, for example, the important situation when $A/J$ is not only semisimple but finite dimensional. As consequence of our method, we also give short straightforward proofs of two other results of \cite{L} and \cite{Ko}, stating that Faith's Conjecture is true for countable dimensional algebras, or for rings $A$ for which $|A/J|\leq \aleph_0$ or $|A/J|<|A|$. 

For sake of completeness, we recall a few facts most of which are fairly easy to see and well known in literature. Let $A$ be a ring and $J$ its Jacobson radical. If $A$ is semilocal, i.e. $A/J$ is semisimple, then an $A$-module is semisimple if and only if it is canceled by $J$. Indeed, every simple is canceled by $J$, and if $JM=0$ then $M$ has an $A/J$-module structure which is semisimple; therefore $M$ is semisimple as the lattice of $A$-submodules and $A/J$-submodules of $M$ coincide in this case. If $A$ is semiprimary, and $n$ is such that $J^n=0\neq J^{n-1}$ then $A$ is semiartinian with a Loewy series of length $n-1$ since $J^k/J^{k+1}$ is semisimple for all $k$. Write $A\bigoplus\limits_eAe$ a sum of indecomposable $A$-modules; such a decomposition obviously exists because $A/J$ has finite length, and each $Ae$ is obtained for some indecomposable idempotent $e$. Note that if $A$ is left self-injective, then each indecomposable $Ae$ has simple socle: indeed, if we have a nontrivial decomposition of the socle $s(Ae)=M\oplus N$, then we can find $E(M),E(N)$ injective hulls of $M,N$ contained in $Ae$, and we obtain $Ae=E(M)\oplus E(N)$ a nontrivial decomposition. This is a contradiction. 

Note also that if $A$ is left self-injective semiprimary, for each simple left $A$-module, the right $A$-module $\Hom(S,A)$ is simple. First, note that it is nonzero. For this, we look at the isomorphism types of indecomposable modules $Ae$; these are projective and local, and are the cover of some simple $A$-module. They are isomorphic if and only if their respective "tops" are isomorphic. The number of isomorphism types of such modules equals the number of isomorphism types of simple modules $t$. Moreover, since the indecomposable $Ae$'s are also injective with simple socle, we see that they are isomorphic if and only if their socle is isomorphic. This shows that the distinct types of isomorphism of simples occurring as socle of some $Ae$ is also $t$, and so each simple $S$ must appear as socle of some $Ae$ (i.e. it embeds in $A$). This shows that $\Hom(S,A)\neq 0$ for each simple $A$-module $S$. If $f,g\in\Hom(S,A)$, and $f\neq 0$, then $f:S\rightarrow A$ is a mono and since ${}_AA$ is injective there is some $h:A\rightarrow A$ such that $h\circ f=g$. If $f(x)=xc,\,\forall x\in A$, we get $f(x)c=g(x)$ i.e. $f\cdot c=g$ in $\Hom(S,A)$. This shows that $\Hom(S,A)$ is generated by any $f\neq 0$, so it is simple. In particular, since each simple module embeds in $A$ which is left injective, it follows that $A$ is an injective cogenerator of the category of left $A$-modules, i.e. it is a left PF (pseudo-Frobenius) ring. It is easy to see that the same conclusions follow in case $A$ is semilocal, left semiartinian and left selfinjective.

\section{The Main Result}\label{s1}

Let $\Ss$ be a set of representatives for the simple left $A$-modules, $t=|\Ss|$ and $A/J=\bigoplus\limits_{S\in\Ss}S^{n_S}$. Let $\Sigma=s({}_AA)$ be the left socle of $A$. It is easy to see that this is an $A$-sub-bimodule of $A$. Note that since each indecomposable module $Ae$ has simple socle, we have that $length(\Sigma)$ equals the number of terms in the indecomposable decomposition $A=\bigoplus\limits_eAe$, which equals $length({}_AA/J)$ since each indecomposable $Ae$ is local. Let $\Sigma=\bigoplus\limits_{S\in\Ss}S^{p_s}$. We have $\sum\limits_{S\in\Ss}p_S=\sum\limits_{S\in\Ss}n_S$.

\begin{proposition}
Let $A$ be left self-injective and semiprimary. Then the set $\{\Hom(S,A)| S\in \Ss\}$ is a set of representatives for the simple right $A$-modules. In particular, $\Hom(S,A),\Hom(T,A)$ are non-isomorphic for non-isomorphic $S,L\in\Ss$.
\end{proposition}
\begin{proof}
Since $A$ is left injective, the monomorphism $0\rightarrow \Sigma\rightarrow A$ gives rise to the epimorphism of right $A$-modules $\Hom(A,A)\rightarrow \Hom(\Sigma,A)\rightarrow 0$. Note that $\Hom(\Sigma,A)=\bigoplus\limits_{S\in\Ss}\Hom(S,A)$.
Since $\Hom(S,A)\neq 0$ for each $S\in\Ss$ we have $\Hom(\Sigma,A)=\bigoplus\limits_{S\in\Ss}\Hom(S,A)^{p_S}$ has length equal to $length(\Sigma)=\sum\limits_{S\in\Ss}{p_S}=length({}_AA/J)$. But by the classical Wedderburn-Artin theorem, $length{}_A(A/J)=length(A/J)_A$. Since $\Hom(\Sigma,A)$ is semisimple, the kernel of $A\rightarrow \Hom(\Sigma,A)$ contains $J$, and furthermore since $length(A/J)=length(\Hom(\Sigma,A))$, we obtain $A/J\cong \Hom(\Sigma,A)$ as right $A$-modules. This shows that all types of isomorphism of right $A$-modules are found among components of $\Hom(\Sigma,A)$, and so the statement is proved. 
\end{proof}

We note that the above proof further shows that there is an exact sequence of right $A$-modules 
$$0\longrightarrow J\longrightarrow A\longrightarrow \Hom(\Sigma,A)\longrightarrow 0$$
But it is immediate to see that this means that $\{a\in A|\Sigma\cdot a=0\}=J$, i.e. $ann(\Sigma_A)=J$. In particular, this shows that $\Sigma$ is also semisimple as a right $A$-module, i.e. the left socle of $A$ is contained in the right socle. In fact, it is known that the left and right socles of a left PF-ring coincide \cite[Theorem 6]{Ka}, and if $A$ is semiprimary with same left and right socle, then it is easy to show that the left and right Loewy series of $A$ coincide \cite[Proposition 2.1]{AP} (see also \cite[Lemma 3.7]{Ko}). 

For a left $A$-module $M$, let us denote for short $M^*=\Hom(M,A)$; this is a right $A$-module. 

\begin{proposition}\label{p.1}
Let $A$ be a left self-injective ring and let $M$ be a left $A$-module such that there is an exact sequence $0\rightarrow S\rightarrow M\rightarrow L^{(\alpha)}\rightarrow 0$, with $S,L$ simple modules, and assume $S=s(M)$ the socle of $M$, and $L^{(\alpha)}$ denoting the coproduct of $\alpha$ copies of $L$. Then $M^*$ is a local right module with unique maximal ideal $S^\perp=\{f\in \Hom(M,A)|f_{\vert S}\neq 0\}$ which is semisimple isomorphic to $(L^*)^\alpha$.
\end{proposition}
\begin{proof}
We have an exact sequence $0\rightarrow (L^*)^\alpha\rightarrow M^*\rightarrow S^*\rightarrow 0$; it is easy to see that the kernel of the morphism $M^*=\Hom(M,A)\rightarrow S^*=\Hom(S,A)$ is $S^\perp$. Hence $S^\perp\cong (L^*)^\alpha$ which is right semisimple since it is canceled by $J$. Now since $M$ has simple socle, and its socle embeds in $A$ which is injective, it follows that $M$ embeds in $A$. We now note that $M^*$ is generated by any $f\not\in S^\perp$, which will show that $M^*$ is . Indeed, such an $f$ must be a monomorphism, and given any other $h:M\rightarrow A$, by the injectivity of ${}_AA$ there is $g\in\Hom(A,A)$ such that $g\circ f=h$. If $g(x)=xc$ for $c\in A$ then we have $h=f\cdot c$ in $M^*$. This shows that $f\cdot A=M^*$. This obviously shows that $S^\perp$ is the only maximal submodule of the cyclic right $A$-module $M^*$.
\end{proof}

Note that the fact that $M^*$ is local can also be proved by embedding $M$ in some indecomposable $Ae$ for an indecomposable idempotent $e$, and then, by applying the exact functor $\Hom(-,A)$, one obtains an epimorphism $\Hom(Ae,A)=eA\rightarrow M^*$, and so $M^*$ is local because $eA$ is.

Let $\alpha$ be the largest cardinality for which there is a left module $M$ with simple socle and such that $M/s(M)\cong L^{\alpha}$ for some simple module $L$. Such a cardinality obviously exists, since any such module is contained in $A$ because $A$ is injective. In fact, if $\Sigma_1$ is the second socle of $A$, then $\alpha\leq length(\Sigma_2/\Sigma)$. We note that if $\alpha$ is infinite, this is an equality. Indeed, if for each simple modules $S,L$ we denote by $\alpha_{S,L}=[E(S)/S):L]$ - the multiplicity of $L$ in the second socle of the injective hull $E(S)$ of $S$, then $\alpha=\max_{S,L\in\Ss}\alpha_{S,L}$. Therefore, $\alpha\leq\sum\limits_{S,L\in\Ss}\alpha_{S,L}\leq n\alpha=\alpha$ if $\alpha$ is infinite. We note also that if $\Sigma_k$ is the $k$'th socle, then $length(\Sigma_k/\Sigma_{k-1})\leq \alpha$; this follows by induction on $k$: if this is true for $k$, then there is an embedding $\Sigma_k/\Sigma_{k-1}\hookrightarrow A^{(\alpha)}$, and therefore we have $length(\Sigma_{k+1}/\Sigma_k)\leq length(\Sigma_1/\Sigma_0)^{(\alpha)}=\alpha\times \alpha=\alpha$ since $\alpha$ is an infinite cardinal. We therefore have

\begin{theorem}\label{t.1}
Let $A$ be a self-injective semiprimary algebra such that the dimension of each simple $A$-module is at most countable (equivalently, the dimension of $A/J$ is at most countable).
\end{theorem}
\begin{proof}
With the above notations, assume $\alpha$ is infinite. The length of each $\Sigma_k/\Sigma_{k-1}$ is at most $\alpha$, so since the dimension of each simple is at most $\aleph_0$, its dimension is at most $\aleph_0\times \alpha=\alpha$ (since $\alpha$ is infinite). Thus, the dimension of $A$ is at most $\alpha$, and so it equals $\alpha$ (since $length(\Sigma_1/\Sigma_0)=\alpha$). On the other hand, by Proposition \ref{p.1}, there is a local right $A$-module $M^*$, with socle $L^\alpha$ for some simple $L$. Note that $\dim(L^\alpha)\geq 2^\alpha$, and that there is an epimorphism $A\rightarrow M^*$, so $\dim(A)\geq 2^\alpha$. This is obviously a contradiction. 
\end{proof}

We note that can also prove this by using the exact sequences $0\rightarrow \Sigma_k/\Sigma_{k-1}\rightarrow A/\Sigma_{k-1}\rightarrow A/\Sigma_k\rightarrow 0$, which by the left injectivity of $A$ yield the exact sequences of right $A$-modules $0\rightarrow (\Sigma_k/\Sigma_{k-1})^*\rightarrow (A/\Sigma_{k-1})^*\rightarrow (A/\Sigma_k)^*\rightarrow 0$ so $\dim(A/\Sigma_{k-1})^*-\dim(A/\Sigma_k)^*=\dim(\Sigma_k/\Sigma_{k-1})^*$, which, by summing for $k$ yields 
$$\dim(A)=\sum\limits_k\dim(\Sigma_k/\Sigma_{k-1})^*.$$ 
Then one can proceed as above to note that in the situation when $\alpha$ is infinite and $\dim(A/J)$ is at most countable, then one of the dimensions on the right of the above equalities is at least $2^\alpha$, while $\dim(A)=\alpha$, a contradiction.

We note several other corollaries that can be obtained applying the above method. The following can also be obtained from the results of \cite{L}, which shows that a left self-injective at most countable dimensional algebra is QF; nevertheless, the proofs of \cite{L} use some further assumptions on $A$, such that the cardinality of $A$ is regular, and also makes use of the generalized continuum hypothesis (see also MR512077, Erratum to: [\emph{A countable self-injective ring is quasi-Frobenius}, Proc. Amer. Math. Soc. 65
(1977), no. 2, 217–-220]; Proc. Amer. Math. Soc. 73 (1979), no. 1, 140).

\begin{corollary}
A semiprimary left self-injective algebra of countable dimension then is QF. 
\end{corollary}

The following is known from \cite[Corollary 3.10]{Ko}. We also give a very short (and straightforward) proof of this using the method above. 

\begin{proposition}
(1) A left self-injective semiprimary ring $A$ with $|A/J|\leq \aleph_0$ is QF.\\
(2) A left self-injective semiprimary ring $A$ with $|A/J|<|A|$ is QF.
\end{proposition}
\begin{proof}
(1) We proceed as in Theorem \ref{t.1}, and keep the notations above. The length of each $\Sigma_k/\Sigma_{k-1}$ is at most $\alpha$, and since each simple module has cardinality at most $\aleph_0$, $|\Sigma_{k}/\Sigma_{k-1}|\leq \aleph_0\times \alpha=\alpha$. As in Theorem \ref{t.1}, using Proposition \ref{p.1} we find the right module $M^*$ with socle $L^\alpha$, which has cardinality at least $2^{\alpha}$, and is a quotient of $A$. This yields a contradiction.\\
(2) Let $c$ be the largest cardinality of a simple left $A$-module; we have $c<|A|$. Again, as above, if $\alpha$ is infinite, we obtain that the cardinality of $A$ has to be at least $2^{\alpha}$. On the other hand, the cardinality of the modules $\Sigma_k/\Sigma_{k-1}$ is less than $c\times \alpha={\rm max}(c,\alpha)$. Since $\Sigma_n=A$ for some $n$, this shows that $|A|\leq \max{c,\alpha}$. But $\alpha<2^\alpha\leq |A|$ and $c<|A|$, yielding a contradiction. 
\end{proof}

We now note another interesting fact about the general situation of Faith's QF conjecture.

\begin{proposition}
Let $A$ be a left self-injective semiprimary ring. Then \\
(i) each right $A$-module $eA$ has simple socle. Consequently, the right socle of $A$ (which coincides with $\Sigma$) is finitely generated, and has the same left and right lengths.\\
(ii) $\Hom(T,A)\neq 0$ for each right simple $A$-module $T$.
\end{proposition}
\begin{proof}
(i) We have $eA=\Hom(Ae,A)$. Let $M$ be the unique maximal submodule of $Ae$. We show that $M^\perp=\{f:Ae\rightarrow A| \,f\vert_M=0\}\subset eA$ is essential in $A$. Let $0\neq h:Ae\rightarrow A$. Then, $\ker(h)\neq Ae$, so $\ker(h)\subseteq M$, and thus we have the following commutative diagram
$$\xymatrix{
 & Ae \ar[d]_p \ar[dr]^h & \\
0 \ar[r] & \frac{Ae}{\ker(h)}\ar[d]_\pi \ar[r]^i & A \ar@{..>}[ddl]^g \\ 
& \frac{A}{M}\ar[d]_u &\\
& A &
}$$
Here, $p$ and $\pi$ are the canonical projections, $h=i\circ p$ is the canonical decomposition, and $u$ is a nonzero morphism from $A/M$ to $A$, which exists since we know all isomorphism types of simple modules embed in $A$. Since $A$ is injective, then the above diagram is completed commutatively by a $g:A\rightarrow A$, $g(x)=xc$ for $x\in A$. Let $f=u\circ\pi\circ p$; then obviously $f\neq 0$, $f\in M^\perp$ and $g\circ h=f$, i.e. $g\cdot c=f$ in $\Hom(Ae,A)$. This shows that $M^\perp \cap hA\neq 0$ whenever $h\neq 0$. This shows that $M^\perp$ is essential in $eA$. Moreover, it is easy to see that $M^\perp\cong (Ae/M)^*$ by dualizing the exact sequence $0\rightarrow M\rightarrow Ae\rightarrow Ae/M\rightarrow 0$, so $M^\perp$ is simple. Thus, $eA=\Hom(Ae,A)$ has simple (essential) socle.\\
(ii) We have already noticed that each simple right $A$-module $T$ is of the form $\Hom(S,A)$ for a simple left $A$-module. But since there is an epimorphism $A\rightarrow S$, by duality we get a monomorphism of right $A$-modules $0\rightarrow T=\Hom(S,A)\rightarrow A$.
\end{proof}

We note that a possible procedure for proving this conjecture for other cases, would be the following. For a semiprimary left self-injective ring $A$, consider the Loewy series $0\subset \Sigma_0\subset \dots \Sigma_k\dots \Sigma_{n-1}=A)$ of $A$ - this is the same to the left and to the right. The first term has the same left and right length, as shown before. If this is true for all the factors in the Loewy series, that is, if the length of $\Sigma_{k}/\Sigma_{k-1}$ would be the same as left and right modules, one could apply the above procedure of Proposition \ref{p.1} and Theorem \ref{p.1} to obtain a positive answer to Faith's Conjecture (in fact it is enough to show that $\Sigma_1/\Sigma_0$ has the same left and right length). Specifically, let $M$ be a module like in Proposition \ref{p.1} of maximal (infinite) length $\alpha$ modulo its socle; one then sees that $M^*/M^*J$ has $A$-length greater than $\alpha$ (in fact, it is semisimple of length $2^\alpha$), by regarding everything as vector spaces over some division algebra and using methods similar to those of vector spaces. This would again be a contradiction to the fact that the left and right lengths of $\Sigma_{n-1}/\Sigma_{n-2}$ coincide. On the other hand, this also shows that if a counterexample to this conjecture exists, then some $\Sigma_{k}/\Sigma_{k-1}$ would have different left and right lengths.

One way one could try to compare the left and right lengths of $\Sigma_{k}/\Sigma_{k-1}$ is to decompose it into a direct sum of bimodules $\bigoplus\limits_{S,T}M_{S,T}$, with each $M_{S,T}$ being left and right semisimple and iso-typical (isomorphic to powers of some $S$ as a left module and some $T$ as a right module). Then one essentially needs to compare left and right lengths of certain $\Delta-\Delta'$-bimodules for some division algebras $\Delta$ and $\Delta'$. 

\bigskip\bigskip\bigskip



\end{document}